\documentclass[11pt,  a4paper]{amsart}
\usepackage{eucal}
\usepackage{graphicx}
\usepackage[english]{babel}
\usepackage{amsmath,amsthm,amssymb,amsfonts,amscd,amsbsy,amsxtra}
\usepackage{epsfig,color}
\usepackage{subfigure}
\usepackage{psfrag}
\usepackage[arrow,matrix, curve]{xy}
\usepackage{amsmath}
\usepackage{enumerate}
\usepackage[T1]{fontenc} 
\usepackage[utf8]{inputenc}
\usepackage{mathrsfs}
\usepackage[shortlabels]{enumitem}
\usepackage{mathtools}
\usepackage[bookmarks=true,pdfstartview=FitH, pdfborder={0 0 0}, colorlinks=true,citecolor=red, linkcolor=blue]{hyperref}
\usepackage{stmaryrd}
\usepackage{bbm}
\usepackage{aliascnt}
\usepackage{a4wide}
\usepackage{nicefrac}

\newcommand {\R}	{\mathbb{R}}
\newcommand {\N}	{\mathbb{N}}
\newcommand {\Z}	{\mathbb{Z}}

\newcommand {\C}	{\mathbb{C}}

\DeclareMathOperator{\re}{Re}

\DeclareMathOperator{\Id}{Id}
\DeclareMathOperator{\rg}{rg}

\DeclareMathOperator{\dvol}{dvol}
\DeclareMathOperator{\loc}{{loc}}

\renewcommand{\div}{\mathrm{div}}

\newcommand{\dX}{{\partial X}}

\newcommand{\sL}{\mathcal{L}}

\newcommand{\rC}{\mathrm{C}}
\newcommand{\rL}{\mathrm{L}}
\newcommand{\rW}{\mathrm{W}}

\newcommand{\ddn}{\frac{\partial^a}{\partial \nu^g}}

\renewcommand{\epsilon}{\varepsilon}

\advance\textheight by 2cm

\setlist[enumerate]{font = \normalfont}

\theoremstyle{plain}
\newtheorem{thm}{Theorem}[section]
\newaliascnt{cor}{thm}
\newaliascnt{prop}{thm}
\newaliascnt{lem}{thm}
\newtheorem{cor}[cor]{Corollary}
\newtheorem{prop}[prop]{Proposition}
\newtheorem{lem}[lem]{Lemma}
\aliascntresetthe{cor}
\aliascntresetthe{prop}
\aliascntresetthe{lem} 
\newcounter{stp}
\newcounter{stpi}
\newcounter{stpci}
\newcounter{stpiii}

 \theoremstyle{theorem}
\newtheorem{step}[stp]{Step}

\theoremstyle{definition}
\newaliascnt{defn}{thm}
\newaliascnt{asu}{thm}
\newaliascnt{con}{thm}
\newtheorem{defn}[defn]{Definition}

\aliascntresetthe{defn}
\aliascntresetthe{asu}
\aliascntresetthe{con}
\theoremstyle{remark}
\newaliascnt{rem}{thm}
\newaliascnt{exa}{thm}
\newaliascnt{masu}{thm}
\newaliascnt{nota}{thm}
\newaliascnt{sett}{thm}
\newtheorem{rem}[rem]{Remark}

\newtheorem{nota}[nota]{Notation}
\newtheorem{sett}[sett]{Abstract Setting}
\aliascntresetthe{rem}
\aliascntresetthe{exa}
\aliascntresetthe{masu}
\aliascntresetthe{nota}
\aliascntresetthe{sett}
\numberwithin{equation}{section}

\title [Dirichlet-to-Neumann operators on manifolds]
{Dirichlet-to-Neumann operators 
	on manifolds 
}
\author{Tim Binz}

\date{\today}%

\begin{document}

\maketitle

\begin{abstract}
We consider the Dirichlet-to-Neumann operator associated to a strictly elliptic operator on the space $\rC(\partial M)$ of continuous functions on the boundary $\partial M$ of a compact manifold $\overline{M}$ with boundary. We prove that it generates an analytic semigroup of angle $\nicefrac{\pi}{2}$, generalizing and improving \cite{Esc:94} with a new proof. Our result fits with the main result in \cite{EO:17} in the case of domains with smooth boundary. Combined with \cite[Thm.~3.1]{EF:05} and \cite{Bin:18a} this yields that the corresponding strictly elliptic operator with Wentzell boundary conditions generates a compact and analytic semigroups of angle $\nicefrac{\pi}{2}$ on the space $\rC(\overline{M})$.
\end{abstract}

\section{Introduction}

Differential operators with dynamic boundary conditions on manifolds with boundary describe a  system whose dynamics consisting of 
two parts: a dynamics on the manifold interacting with an additional dynamics on the boundary. This leads to differential operators with so called Wentzell boundary conditions, see \cite[Sect.~2]{EF:05}.

\smallskip
On spaces of continuous functions on domains in $\R^n$ such operators have first been studied systematically by 
Wentzell \cite{Wen:59} and Feller \cite{Fel:54}.
Later Arendt et al. \cite{AMPR:03} proved that the Laplace operator with Wentzell boundary conditions generates a positive, contractive $C_0$-semigroup. Engel \cite{Eng:03} improves this by showing that this semigroup is analytic with angle of analyticity $\nicefrac{\pi}{2}$. 
Later Engel and Fragnelli \cite{EF:05} generalize this result to uniformly elliptic operators, however without specifying the corresponding angle of analyticity. For related work see also
\cite{CT:86}, \cite{CM:98alt}, 
\cite{FGGR:02}, 
\cite{CENN:03}, 
\cite{VV:03}, 
\cite{CENP:05}, 
\cite{FGGR:10},
\cite{War:10}
and the references therein. Our interest in this context is the generation of an analytic semigroup with the optimal angle of analyticity.

\smallskip
As shown in \cite{EF:05} and \cite{BE:18} this problem is closely connected to the generation of an analytic semigroup by the Dirichlet-to-Neumann operator on the boundary space.
More precisely, based on the abstract theory for boundary perturbation problems developed by Greiner in \cite{Gre:87}, it has been shown in \cite{EF:05} and in \cite{BE:18} that the coupled dynamics 
can be decomposed into two independent parts: a dynamics on the interior and a dynamics on the boundary. The first one is described by the differential operator on the manifold with Dirichlet boundary conditions while the second is governed by the associated Dirichlet-to-Neumann operator. 

\smallskip
On domains in $\R^n$ the generator property of differential operators with Dirichlet boundary conditions is quite well understood, see 
\cite{Ama:95} and \cite{Lun:95}. 
On compact Riemannian manifolds with boundary it has been shown in \cite{Bin:18a} that strictly elliptic operators with Dirichlet boundary conditions are sectorial of angle $\nicefrac{\pi}{2}$ and have compact resolvents on the space of continuous functions. 

\smallskip
Dirichlet-to-Neumann operators have been studied e.g. by \cite{SU:90}, \cite{LU:01}, \cite{LTU:03} and \cite[App.~C]{Tay:81}.
For the operator-theoretic context see, e.g., the work of
Amann and Escher \cite{AE:96} and Arendt and ter Elst \cite{AT:11}, \cite{ATES:13} and \cite{AE:17}. In particular, on domains in $\R^n$ Escher \cite{Esc:94} has shown that such Dirichlet-to-Neumann operators generate analytic semigroups on the space of continuous functions, however without specifying the corresponding angle of analyticity. Finally, ter Elst and Ouhabaz \cite{EO:17} proved that this angle is $\nicefrac{\pi}{2}$ and extended the result of Escher \cite{Esc:94} to differential operators with less regular coefficients. 

\smallskip
In this paper we study such Dirichlet-to-Neumann operators on the space of continuous functions on Riemannian manifolds 
and show that they generate compact and analytic semigroups of angle $\nicefrac{\pi}{2}$ on the continuous functions. 

\smallskip
We first explain our setting and terminology. 
Consider a strictly elliptic differential operator $A_m: D(A_m) \subset \rC(\overline{M}) \to \rC(\overline{M})$, as given in \eqref{Def: A_m M}, on the space $\rC(\overline{M})$ of continuous functions on a smooth, compact, orientable Riemannian manifold $\overline{M}$ with smooth boundary $\partial M$. Moreover, let $\ddn:D(\ddn)\subset \rC(\overline{ M}) \to \rC(\partial {M})$ be the outer conormal derivative, $\beta>0$ and $\gamma\in \rC(\partial M)$. We consider $B := -\beta\cdot \ddn f+\gamma\cdot f\big|_{\partial M} :D(B) \subset \rC(\overline{M}) \to \rC(\partial M)$, as in \eqref{Def: B M}, and define the operator $A^B f := A_m f$ with \emph{Wentzell boundary conditions} by requiring
\begin{equation}\label{eq:bc-W-Lap} 
f\in D(A^B)
\quad:\iff\quad f \in D(A_m) \cap D(B) \text{ and }
A_m f\big|_{\partial M}=Bf .
\end{equation}
For a continuous function $\varphi \in \rC(\partial M)$ on the boundary the corresponding Dirichlet problem
\begin{align}
\begin{cases}
A_m f = 0, \label{eq:Dir-Prob} \\
f|_{\partial M} = \varphi, 
\end{cases}
\end{align}
is uniquely solvable by \cite[Cor.~9.18]{GT:01}.  Moreover, by  the maximum principle, see \cite[Thm.~9.1]{GT:01}, the associated solution operator $L_0:\rC(\partial {M})\to\rC(\overline{ M})$ is bounded. 
Then the \emph{Dirichlet-to-Neumann operator} is 
\begin{equation}\label{def-N}
N \varphi := -\beta \ddn\cdot L_0 \varphi\quad \text{ for } \varphi \in 
D(N) := \left\{ \varphi \in \rC(\partial M) \colon L_0 \varphi \in D(B) \right\}.
\end{equation}
That is, $N\varphi$ is obtained by applying the Neumann boundary operator $-\beta \ddn$ to the solution $f$ of the Dirichlet problem \eqref{eq:Dir-Prob}. 

%
%
\smallskip 

Our main results are the following. 
\begin{enumerate}[(i)]
\item The Dirichlet-to-Neumann operator
$N$ in \eqref{def-N} generates a compact and analytic semigroup of angle $\nicefrac{\pi}{2}$ on $\rC(\partial M)$;
\item the operator
$A^B$ with Wentzell boundary conditions \eqref{eq:bc-W-Lap}  generates a compact and analytic semigroup of angle $\nicefrac{\pi}{2}$ on $\rC(\overline{M})$.
\end{enumerate}

This extends the results from Escher \cite{Esc:94} and Engel-Fragnelli \cite[Cor.~4.5]{EF:05} to elliptic operators on compact manifolds with boundaries and gives the maximal angle of analyticity $\nicefrac{\pi}{2}$ in both cases. In the flat case the result for the Dirichlet-to-Neumann operator coincides with the result of ter Elst-Ouhabaz \cite{EO:17} in the smooth case. The techniques here are different and our proof is independent from theirs. The compactness and the analyticity of angle $\nicefrac{\pi}{2}$ of the semigroup imply that the spectra $\sigma(N)$ and $\sigma(A^B)$ consist of real eigenvalues only.  

\smallskip

This paper is organized as follows. In Section~2 below we recall the abstract setting from \cite{EF:05} and \cite{BE:18} needed for our approach. Based on \cite[Sect.~2]{Eng:03}, we study in Section~3
the special case where $A_m$ is the Laplace-Beltrami operator and $B$ the normal derivative.
In Section~4 we then generalize these results to arbitrary strictly elliptic operators and their conormal derivatives. Moreover, we use this
to obtain uniqueness, existence and estimates for the solutions of the Robin-Problem. Here the main idea is to introduce a new
Riemannian metric induced by the coefficients of the second order part of the elliptic operator. Then the operator takes a simpler
form: Up to a relatively bounded perturbation of bound $0$, it coincides with 
a Laplace-Beltrami operator for the new metric. 
Regularity and perturbation theory for operator semigroups as in \cite[Sect.~4]{BE:18} then yield the first part of the main theorem in its full generality. The second part follows from \cite[Thm.~3.1]{EF:05} and \cite[Thm.~1.1]{Bin:18a}.

\medskip

In this paper the following notation is used. 
For a closed operator $T \colon D(T) \subset X \rightarrow X$ on a Banach space $X$ we denote by $[D(T)]$ the Banach space $D(T)$ 
equipped with the graph norm $\| \bullet \|_T := \| \bullet \|_X + \| T (\bullet) \|_X$ and indicate by $\hookrightarrow$ a continuous
and by $\stackrel{c}{\hookrightarrow}$ a compact embedding. 
Moreover, we use Einstein's notation of sums, i.e., 
\begin{align*}
x_k y^k := \sum_{k = 1}^n x_k y^k
\end{align*}
for $x := (x_1, \dots, x_n), y := (y_1, \dots, y_n)$.

\section*{Acknowledgments}

The author wishes to thank Professor Simon Brendle and Professor Klaus-J.\ Engel for many helpful suggestions and discussions. 

\section{The abstract Setting}

The starting point of our investigation is the abstract setting proposed first in this form by \cite{Gre:87} and successfully used, e.g., in
\cite{CENN:03}, 
\cite{CENP:05} and \cite{EF:05} for the study of boundary perturbations.

\begin{sett}\label{set:AS}
Consider
\begin{enumerate}[(i)]
\item two Banach spaces $X$ and $\dX$, called \emph{state} and 
\emph{boundary space}, respectively;
\item a densely defined \emph{maximal operator}
$A_m \colon D(A_m) \subset X \rightarrow X$;
\item a \emph{boundary (or trace) operator} $L \in \sL(X,\dX)$;
\item a \emph{feedback operator} $B \colon D(B) \subseteq X \rightarrow \dX$.
\end{enumerate}
\end{sett}

Using these spaces and operators we define the operator $A^B:D(A^B)\subset X\to X$
with \emph{generalized Wentzell boundary conditions} by
\begin{equation}\label{eq:W-BC}
A^B f := A_m f, \quad 
D(A^B):= \bigl\{ f \in D(A_m) \cap D(B) : LA_mf = Bf \bigr\} . 
\end{equation}

For our purpose we need  some more 
operators. 

\begin{nota} 
We denote the (closed) kernel of $L$ by $X_0 := \ker(L)$ and
consider the restriction $A_0$ of $A_m$ given by
\begin{alignat*}{3}
&A_{0}:D(A_0)\subset X\to X,&\quad&D(A_{0}) := \bigl\{ f \in D(A_m) : Lf = 0 \bigr\}. 
\end{alignat*}
\end{nota}

The \emph{abstract
Dirichlet operator associated with $A_m$}
is, if it exists, 
\begin{equation*}
L^{A_m}_0 := \bigl(L|_{\ker(A_m)}\bigr)^{-1} \colon \dX
\rightarrow \ker(A_m) \subseteq X,
\end{equation*}
i.e., $L^{A_m}_0 \varphi = f$ is equal to the solution of the abstract Dirichlet problem
\begin{equation}
\begin{cases}
A_m f = 0, \\
Lf = \varphi .
\end{cases}\label{Dirichlet Problem}
\end{equation}
If it is clear which operator $A_m$ is meant, we simply write $L_0$.

\smallskip

Moreover
for $\lambda \in \C$ we define the \emph{abstract
Robin operator associated with $(\lambda,A_m,B)$}
by 
\begin{equation*}
R^{A_m,B}_\lambda := \bigl((B-\lambda L)|_{\ker(A_m)}\bigr)^{-1} \colon \dX \rightarrow \ker(A_m) \cap D(B) \subseteq X,
\end{equation*}
i.e., $R^{A_m,B}_\lambda \varphi = f$ is equal to the solution of the abstract Robin problem
\begin{equation}
\begin{cases}
A_m f = 0, \\
Bf - \lambda Lf = \varphi .
\end{cases}\label{Robin Problem}
\end{equation}
If it is clear which operators $A_m$ and $B$ are meant, we simply write $R_\lambda$.

\smallbreak
Furthermore, we introduce the \emph{abstract Dirichlet-to-Neumann operator associated with $(A_m,B)$} defined by 
\begin{equation}
N^{A_m, B} \varphi:=BL^{A_m}_0 \varphi,
\quad
D(N^{A_m, B}) := \bigl\{\varphi \in \dX : L^{A_m}_0 \varphi \in D(B) \bigr\}. \label{def:N}
\end{equation}
If it is clear which operators $A_m$ and $B$ are meant, we call $N$ simply the (abstract) Dirichlet-to-Neumann operator.
This Dirichlet-to-Neumann operator is an abstract version of the operators studied  in many places, e.g., \cite{Esc:94}, \cite[Sect.~7.11]{Tay:96}
and \cite[Sect.~II.5.1]{Tay:81}.

The Dirichlet-to-Neumann and the Robin operator are connected in
the following way.

\begin{lem}\label{D-N Resolvente}
If $L_0$ exists, we have $\lambda \in \rho(N^{A_m,B})$ if and only if $R^{A_m, B}_\lambda \in \mathcal{L}(\dX,X)$ exists. If one of these conditions is satisfied, we obtain
\begin{equation*}
	R^{A_m,B}_\lambda = - L_0 R(\lambda, N^{A_m, B}) .
\end{equation*} 
\end{lem}
\begin{proof}
	Assume that $R_{\lambda} \in \mathcal{L}(\dX,X)$ exists.
	By the definition of $N$ the equation
	\begin{align*}
		\lambda \varphi-N \varphi = \psi  
	\end{align*}
	for $\varphi, \psi \in \dX$ is equivalent to
	\begin{align}
		\lambda L L_0 \varphi - B L_0 \varphi = \psi \label{eq:Resolvent D-N}
	\end{align}
	for $\varphi, \psi \in \dX$. This again is equivalent to
	\begin{align*}
		-R_{\lambda} \psi = L_0 \varphi .
	\end{align*}
	Therefore,  we have for $\varphi, \psi \in \dX$ the equivalence
	\begin{align*}
		\mu \varphi - N \varphi = \psi   \quad \Longleftrightarrow \quad 
		R_{\lambda} \psi = -L_0 \varphi .
	\end{align*}
		Since $R_{\lambda,\mu}: \dX\to \ker(A_m)\cap D(B)$ exists and $L_0 : \dX \to \ker(A_m)$ is an isomorphism, there exists a unique $\varphi \in D(N)$ for every $\psi \in \dX$. 
		Moreover its given by $\phi = - L R_{\lambda,\mu} \psi$ and therefore the boundedness of the inverse follows from the boundedness of $L$ and $R_{\lambda}$. The formula for the resolvent of $N$ follows, since $L|_{\ker(A_m)}$ is an isomorphism with inverse $L_0$ and the image of $R_{\lambda}$ is contained in $\ker(A_m)$. 
		\medskip 	
		
		Conversely, we assume that $\mu \in \rho(N)$. Then \eqref{eq:Resolvent D-N} has a unique solution $\varphi \in D(N)$ for every $\psi \in \dX$. 
		Considering $f := -L_0 \varphi$ we obtain a unique solution of \eqref{Robin Problem} and hence $R_{\lambda}$ exists. Boundedness follows from $R_{\lambda} = - L_0 R(\mu,N)$.  
\end{proof}

\section{Boundary problems for the Laplace-Beltrami operator}

In order to obtain a concrete realization of the above abstract objects we consider a smooth, compact, orientable Riemannian manifold $(\overline{M},g)$ with smooth boundary $\partial M$, where $g$ denotes the Riemannian metric.
Moreover, we take the Banach spaces $X := \rC(\overline{M})$ and $\partial X = \rC(\partial M)$
and as the maximal operator the Laplace-Beltrami operator 
\begin{equation}
A_m f := \Delta_M^{g} f,\quad
D(A_m) := \biggl\{ f \in \bigcap_{p > 1} \rW^{2,p}_{\loc}(M) \cap \rC(\overline{M}) \colon \Delta_{M}^g f \in \rC(\overline{M}) \biggr\}.
\label{L-B}
\end{equation}
As feedback operator we take the normal derivative 
\begin{equation}
B f := - g\bigl( \nabla_M^g f , \nu_g\bigr), \quad
D(B) := \biggl\{ f \in \bigcap_{p > 1} \rW^{2,p}_{\loc}(M) \cap \rC(\overline{M}) \colon B f \in \rC(\partial M) \biggr\},
\label{nDeriv}
\end{equation}
where $\nabla_M^g$ denotes the gradient on $M$, which in local coordinates is given as
\begin{equation*}
\bigl(\nabla_M^g f\bigr)^l = g^{kl} \partial_k f 
\end{equation*}
for $f \in \bigcap_{p > 1} W^{1,p}(M)$. 
Moreover, $\nu_g$ is the outer normal on $\partial M$ given in local coordinates by
\begin{equation*}
\nu_g^l = g^{kl}\nu_k  .
\end{equation*}
Furthermore, we choose $L$ as the trace operator, i.e.,
\begin{equation*}
L \colon X \rightarrow \dX,\  f \mapsto f|_{\partial M},
\end{equation*}
which is bounded with respect to the supremum norm.
Later on we will also need the unique bounded extension of $L$ to $\rW^{1,2}(M)$, denoted by $\overline{L} \colon \rW^{1,2}(M) \rightarrow \rL^2(\partial M)$, and call it the 
(generalized) trace operator.

\subsection{The Laplace-Beltrami operator with Robin boundary conditions}
\

In this setting we consider the Laplace-Beltrami operator with Robin boundary conditions
and prove existence, uniqueness and regularity for the solution of \eqref{Robin Problem}. 
Moreover, we show that this solution satisfies a maximum principle. 

\smallskip 

For this purpose we need the concept of a weak solution of \eqref{Robin Problem}.
If $f \in D(A_m)\cap D(B)$ is a solution of \eqref{Robin Problem} we obtain by Green's Identity
\begin{align*}
\int_M g\bigl(\nabla_M^g f , \nabla_M^g \overline{\phi}\bigr) \, \dvol_M^g 
= - \int_{\partial M} B f \overline{L}\overline{\phi} \, \dvol_{\partial M}^g
= - \int_{\partial M} \lambda \overline{L} f \overline{L} \overline{\phi} \, \dvol_{\partial M}^g 
- \int_{\partial M} \varphi \overline{L} \overline{\phi} \dvol_{\partial M}^g
\end{align*}
for all $\phi \in \rW^{1,2}(M)$. This motivates the following definition. 

\begin{defn}[Weak solution of the Robin Problem]
We call $f \in \rW^{1,2}(M)$ a \emph{weak solution of \eqref{Robin Problem}} if it satisfies
\begin{equation*}
\mathfrak{a} (f,\phi) := \int_M g\bigl(\nabla_M^g f , \nabla_M^g \overline{\phi}\bigr) \, \dvol_M^g 
+ \int_{\partial M} \lambda \overline{L} f \overline{L} \overline{\phi} \, \dvol_{\partial M}^g
= - \int_{\partial M} \varphi \overline{L} \overline{\phi} \dvol_{\partial M}^g =: F(\phi)
\end{equation*}
for all $\phi \in \rW^{1,2}(M)$. 
\end{defn}

Next we prove the existence of such weak solutions. 

\begin{lem}[Existence and Uniqueness of the weak solution of the Robin problem]\label{ex weak}
For each $\re(\lambda) > 0$ the problem \eqref{Robin Problem} has a unique weak solution. 
\end{lem}
\begin{proof}
We consider $\mathfrak{a}$ and $F$ as defined above. Obviously $\mathfrak{a}$ is sesquilinear and $F$ is linear. By the Cauchy-Schwarz Inequality we have for $f, \phi \in \rW^{1,2}(M)$ that
\begin{align*}
|\mathfrak{a}(f, \phi)| \leq \| \nabla_M^g f \|_{\rL^2(M)} \| \nabla_M^g \phi \|_{\rL^2(M)} + |\lambda| \| \overline{L} f \|_{\rL^2(\partial M)} \|  \overline{L} \phi \|_{\rL^2(\partial M)}
\leq C \| f \|_{\rW^{1,2}(M)} \| \phi \|_{\rW^{1,2}(M)},
\end{align*}
hence $\mathfrak{a} \colon \rW^{1,2}(M) \times \rW^{1,2}(M) \rightarrow \C$ is bounded.
Next we show that $\mathfrak{a}$ is coercive. If not, there
exists a sequence $(u_k)_{k \in \N} \subset \rW^{1,2}(M)$ such that
\begin{align*}
\| u_k \|_{\rW^{1,2}(M)}^2 > k \re\bigl(\mathfrak{a}(u_k, u_k)\bigr)  
\end{align*}
for all $k \in \N$.
We consider
\begin{equation*}
v_k := \frac{v_k}{\| v_k \|_{\rW^{1,2}(M)}} \in \rW^{1,2}(M)
\end{equation*}
and remark that $\|v_k \|_{\rW^{1,2}(M)} = 1$ and therefore
\begin{align*}
\re\bigl(\mathfrak{a}(v_k, v_k)\bigr) < \frac{1}{k} 
\end{align*}
for all $k \in \N$. Since $(v_k)_{k \in \N}$ is bounded, by Rellich-Kondrachov (cf. \cite[Cor.~3.7]{Heb:96}) there exists a subsequence
$(v_{k_l})_{l \in \N}$ converging in $\rL^2(M)$ to $v \in \rL^2(M)$.
On the other hand we have
\begin{equation*}
\| \nabla_M^g v_{k_l} \|_{\rL^2(M)} \leq \re\bigl(\mathfrak{a}(v_{k_l}, v_{k_l})\bigr) < \frac{1}{k_l},
\end{equation*} 
hence $(\nabla_M^g v_{k_l})_{l \in \N}$ converges to $0$ in $\rL^2(M)$. 
This shows $v \in \rW^{1,2}(M)$ and $\nabla_M^g v = 0$. 
Moreover, we obtain
\begin{align*}
\| \nabla_M^g v_{k_l} \|_{\rL^2(M)}  = \int_M g_{ij} g^{ir} g^{js} \partial_r v_{k_l} \partial_s v_{k_l} \dvol_M^g
= \int_M g^{rs} \partial_r v_{k_l} \partial_s v_{k_l} \dvol_M^g
= \| \nabla v_{k_l} \|_{\rL^2(M)} , 
\end{align*}
where $\nabla v_{k_l}$ denotes the covariant derivative of $v_{k_l}$. 
Therefore,
$(v_{k_l})_{l \in \N}$ converges in $\rW^{1,2}(M)$ to $v$ with $\| v \|_{\rW^{1,2}(M)} = 1$. 
Moreover, we have
\begin{equation*}
\| \overline{L} v_{k_l} \|_{\rL^2(\partial M)} < \frac{1}{\re(\lambda) k_l}
\end{equation*}
and therefore
\begin{align*}
\| \overline{L} v \|_{\rL^2(\partial M)} \leq \| \overline{L} v - \overline{L} v_{k_l} \|_{\rL^2(\partial M)} + \| \overline{L} v_{k_l} \|_{\rL^2(\partial M)}
< \frac{1}{\re(\lambda) k_l } + C \| v - v_{k_l} \|_{\rW^{1,2}(M)} \longrightarrow 0
\end{align*}
and hence $\overline{L} v = 0$. Since $\nabla v = 0$, we conclude $v = 0$, which contradicts $\| v \|_{\rW^{1,2}(M)} = 1$. Hence, $\mathfrak{a}$ is coercive. 
Since
\begin{equation*}
| F(\phi) | \leq \| \varphi \|_{\rL^2(\partial M)} \| \overline{L} \phi \|_{\rL^2(\partial M)}
\leq C \| \phi \|_{\rW^{1,2}(\partial M)}
\end{equation*}
for all $\phi \in \rW^{1,2}(M)$ we conclude that $F \colon \rW^{1,2}(M) \rightarrow \C$ is bounded. 
By the Lax-Milgram and Fréchet-Riesz theorems 
it follows
that $\alpha(f,\phi) = F(\phi)$ for all $\phi \in \rW^{1,2}(M)$ has a unique solution $f \in \rW^{1,2}(M)$. 
\end{proof}

Next we prove that every weak solution is even a strong solution.

\begin{lem}[Regularity of the Robin problem]\label{regularity} 
Every weak solution of \eqref{Robin Problem} is a strong solution.
\end{lem}
\begin{proof}
By \cite[Chap.~5., Prop.~1.6]{Tay:96} we have $f \in\rC^2(M) \subset \bigcap_{p > 1} \rW^{2,p}_{\loc}(M)$. 

Therefore, we obtain 
by the fundamental lemma of the calculus of variation 
that $\Delta_{M}^g f = 0$, in particular $\Delta_M^g f \in \rC(\overline{ M})$.
Furthermore we have
\begin{equation*}
Bf = \lambda Lf + \varphi \in \rC(\partial M) .
\qedhere
\end{equation*}
\end{proof}

Summing up we obtain the following. 

\begin{cor}[Existence and Uniqueness of the solution of the Robin problem]\label{Existence solution Robin}
For all $\re(\lambda) > 0$ the problem \eqref{Robin Problem} has a unique solution. 
\end{cor}

We finish this subsection by showing a maximum principle for the Robin problem. 

\begin{lem}\label{Max Robin}
A solution $f \in D(A_m) \cap D(B) \subset X$ of 
\eqref{Robin Problem}
satisfies the maximum principle
\begin{align*}
| \re(\lambda) | \cdot \| f \|_X \leq \| \varphi \|_{\dX}
\end{align*}
for all $\re(\lambda) \geq 0$ and $\varphi \in \dX = \rC(\partial M)$.
\end{lem}
\begin{proof}
We consider a point $p \in \overline{M}$, where $|f|$ and therefore $|f|^2$ assumes its maximum. 
By the interior maximum principle (cf. \cite[Thm.~9.1]{GT:01})
it follows that $p \in \partial M$. Hence,  we have
\begin{align*}
g(p)\bigl(\nabla_M^g |f|^2(p), \nu_g(p)\bigr) \geq 0 .
\end{align*}
From
\begin{align*}
g\bigl(\nabla_M^g |f|^2, \nu_g\bigr) &= g\bigl(\nabla_M^g (f \overline{f}), \nu_g\bigr) = 2 \re g\bigl((\nabla_M^g f)\overline{f}, \nu_g\bigr)
= 2 \re \Bigl(g\bigl((\nabla_M^g f), \nu_g\bigr)\overline{f}\Bigr) \\
&= -2 \re\bigl( (Bf)\overline{f}\bigr) = -2 \re\bigl(\varphi \overline{f}\bigr) - 2 \re(\lambda) |f|^2 ,
\end{align*}
we obtain
\begin{align*}
\re(\lambda) |f|^2(p) \leq -\re\bigl(\varphi(p) \overline{f}(p)\bigr) \leq |\varphi|(p) |f|(p).
\end{align*}
Since $\re(\lambda) \geq 0$, this implies
\begin{equation*}
| \re(\lambda) | \cdot \| f\|_X = |\re(\lambda) | \cdot |f|(p) \leq | \varphi|(p) \leq \| \varphi \|_{\dX} . \qedhere
\end{equation*}
\end{proof}

\subsection{Generator property for the Dirichlet-to-Neumann operator}
\

Now we are able to prove our main result: The Dirichlet-to-Neumann operator generates
a contractive and analytic semigroup of angle $\nicefrac{\pi}{2}$ on $\dX = \rC(\partial M)$. To do so we represent the Dirichlet-to-Neumann operator as a 
relatively bounded perturbation of $-\sqrt{- \Delta_{\partial M}^g}$.  

\bigskip

We first need the existence of the associated Dirichlet operator. 

\begin{lem}\label{L_0 exists}
The Dirichlet operator $L_0 \in \mathcal{L}(\partial X, X)$ exists.
\end{lem}
\begin{proof}
This follows by \cite[Chap.~5. (2.26)]{Tay:96}, \cite[Thm.~9.19]{GT:01} and \cite[Thm~9.1]{GT:01}. 
\end{proof}

Next we prove a first generation result for the Dirichlet-to-Neumann operator.

\begin{prop}\label{N Contractive}
The Dirichlet-to-Neumann operator $N$ defined in \eqref{def:N}
generates a contraction semigroup on $\dX$. 
\end{prop}
\begin{proof}
By elliptic regularity 
theory (cf. \cite[Chap.~5.5. Ex.~2]{Tay:96}), 
we have the inclusions
\begin{equation*}
L_0\rC^2(\partial M) \subset\rC^1(\overline{M}) \subset D(B).
\end{equation*} 
Since $\rC^2(\partial M)$ is dense in $\dX$, $N$ is densely defined. 
By \autoref{D-N Resolvente} and \autoref{Existence solution Robin} 
it follows that the resolvent 
$R(\lambda, N)$ exists for all $\re(\lambda) > 0$.
By the interior maximum principle $L|_{\ker(A_m)} \colon \ker(A_m) \subset X \rightarrow \dX$ is an isometry. 
Therefore, \autoref{D-N Resolvente} and \autoref{Max Robin} imply
\begin{align*}
\bigl\| R(\lambda, N) \varphi \bigr\|_{\dX} \leq \frac{1}{| \re(\lambda) |} \| \varphi \|_{\dX}
\end{align*}
for all $\re(\lambda) > 0$ and $\varphi \in \dX$. Hence, the claim follows by the Hille-Yosida Theorem (cf. \cite[Thm.~II.3.5]{EN:00}).
\end{proof}

Now we prove the main result of this subsection.

\begin{thm}\label{mainthm}
The Dirichlet-to-Neumann operator $N$ given by \eqref{def:N} for \eqref{L-B} and \eqref{nDeriv} generates an analytic semigroup of 
angle $\nicefrac{\pi}{2}$ on $\dX$. 
\end{thm}
We proceed as in the proof of \cite[Thm.~2.1]{Eng:03}.
Let $\overline{N}$ and $\overline{ W}$ be the closure of $N$ and $W$, respectively, in $Y := \rL^2(\partial M)$.
Moreover we need results from the theory of pseudo differential operators. We use the notation from \cite{Tay:81} and denote 
by $\text{OPS}^k(\partial M)$ the pseudo differential operators of order $k \in \Z$ on $\partial M$.

\begin{step}
Then the part $\overline{N}|_{\dX}$ coincides with $N$.
\end{step}
\begin{proof}
By \autoref{N Contractive} the Dirichlet-to-Neumann operator $N$ is densely defined and $\lambda - N$, considered as an operator on $Y$, has dense range
$\rg(\lambda - N) = \dX \subset Y$ for all $\lambda > 0$. 
By Green's Identity 
we have
\begin{align*}
\int_M g\bigl( \nabla_M^g f, \nabla_M^g f\bigr) \, d\text{vol}_{M}
+ \int_M  f \Delta_M f \, d\text{vol}_{M}
= \int_{\partial M} g\bigl(\nabla_M^g f , \nu_g\bigr) Lf
\, d\text{vol}_{\partial M} . 
\end{align*}
Hence, for $f := L^{A_m}_0 \varphi$ with $\varphi \in D(N)$ we obtain
\begin{align*}
0 \leq \int_M g\bigl( \nabla_M^g f, \nabla_M^g f\bigr) \, d\text{vol}_{M}
= - \int_{\partial M} \varphi N\varphi \, d\text{vol}_{\partial M} 
\end{align*}
since $\Delta_{M}^g f = 0$. Hence, $N$ as an operator on $Y$ is
dissipative. 
By the Lumer-Phillips theorem (see \cite[Thm.~II.3.15]{EN:00}) the closure $\overline{ N}$ of $N$ exists and generates a contraction semigroup on $Y$. This implies that on $\dX$ we have
\begin{align*}
(1-N) \subseteq (1- \overline{N})|_{\dX} , 
\end{align*}
where $1- N$ is surjective and $1-\overline{ N}$ is injective on $\dX$. This is possible only if the domains
$D(1- N)$ and $D(1-\overline{ N})$ coincide,
i.e., $\overline{ N}|_{\dX} = N$. 
\end{proof}

\begin{step}
The operator $W := -\sqrt{- \Delta_{\partial M}^g}$ generates an analytic semigroup of angle $\nicefrac{\pi}{2}$ on
$\dX$. 
\end{step}
\begin{proof}
The Laplace-Beltrami operator $\Delta_{\partial M}^g$ generates an analytic semigroup of angle $\nicefrac{\pi}{2}$ on 
$\rC(\partial M) = \dX$. Hence, the assertion follows by \cite[Thm.~3.8.3]{ABHN:11}.
\end{proof}

\begin{step}
The operator $\overline{W} := - \sqrt{ - \overline{ \Delta_{\partial M}^g}}$ satisfies
$W = \overline{W}|_{\partial X}$. 
\end{step}
\begin{proof}
By \cite[Chap.~8, Prop.~2.4]{Tay:81} the space $\rC^\infty(\partial M)$ is a core for
$W$ and by \cite[Prop.~3.8.2]{ABHN:11}
the domain $D(\overline{ \Delta_{\partial M}^g})$ is a core for 
$\overline{ W}$. Hence,  $\rC^\infty(\partial M)$ is a core for $\overline{W}$ and
since $\rC^\infty (\partial M) \subset D(W)$ we obtain that 
$D(W)$ is a core for $\overline{W}$ on $Y$. This implies that $\overline{W}$ is indeed
the closure of $W$ in $Y$. Moreover, we obtain
\begin{align*}
(1-W) \subseteq (1- \overline{W})|_{\dX} , 
\end{align*}
where $1- W$ is surjective and $1-\overline{W}$ is injective on $\dX$. This is possible only if for the domains we have
\begin{align*}
D(1- W) = D(1-\overline{W}),
\end{align*}
i.e., $\overline{W}|_{\dX} = W$. 
\end{proof}

\begin{step}
The domain of $W$ can be compactly embedded into the Hölder continuous functions, i.e.,
$[D(W)] \stackrel{c}{\hookrightarrow}\rC^\alpha(M)$ for all $\alpha \in (0,1)$. 
\end{step}
\begin{proof}
Consider $\overline{R} := (1+\overline{W})^{-1}$. Then, by \cite[Chap.~XII.1]{Tay:81}, $\overline{R} \in \text{OPS}^{-1}(\partial M)$
and
since $\varphi \in \dX = \rC(\partial M)$ 
we have by \cite[Chap.~XI, Thm.~2.5]{Tay:81} that $\overline{R}\varphi \in
\rW^{1,p}(\partial M)$ for all $p > 1$. 
Hence, $D(W) = \overline{R} \rC(\partial M) \subset \rW^{1,p}(\partial M)$. 
Moreover, by Sobolev embedding (see \cite[Chap.~V. and Rem.~5.5.2]{Ada:75})
\begin{align*}
\rW^{1,p}(\partial M) \hookrightarrow \rC(\partial M)
\end{align*}
for $p > n -1$. By the closed graph theorem we obtain
\begin{equation*}
[D(W)] \hookrightarrow \rW^{1,p}(\partial M) 
\end{equation*}
for $p > n -1$. Since Rellich's embedding (see \cite[Thm.~6.2, Part III.]{Ada:75}) implies
\begin{equation*}
\rW^{1,p}(\partial M) \stackrel{c}{\hookrightarrow}\rC^\alpha(\partial M)
\end{equation*}
for $p > \frac{n-1}{1-\alpha}$, the claim follows. 
\end{proof}

\begin{step}
The difference 
$\overline{P} := \overline{N} - \overline{W} \in \text{OPS}^0(\partial M)$ is a pseudo differential operator of order $0$. Moreover,
$\overline{P}$ considered as an operator on $Y$ is bounded. 
\end{step}
\begin{proof}
This follows from \cite[App.~C, (C.4)]{Tay:81} and \cite[Chap.~XI, Thm.~2.2]{Tay:96}.
\end{proof}

\begin{step}
The part
$P := \overline{P}|_{\rC^\alpha(\partial M)} \colon {\rC^\alpha(\partial M)}\rightarrow {\rC^\alpha(\partial M)}$ is bounded. Moreover, the operator $P$ considered on $\dX$ is relatively $W$-bounded with bound $0$.
\end{step}
\begin{proof}
Form \cite[Chap. XI, Thm 2.2]{Tay:81} it follows $P \in \mathcal{L}(\rC^\alpha(\partial M))$. By \textbf{Step 4} we have
\begin{equation}
[D(W)] \stackrel{c}{\hookrightarrow}\rC^\alpha(\partial M)
\hookrightarrow \rC(\partial M) . \label{W compact}
\end{equation}
Therefore, by Ehrling's lemma (cf. \cite[Thm.~6.99]{RR:04}), for every $\varepsilon >0$ there exists a constant $C_\varepsilon >0$ such that
\begin{equation*}
\| \varphi \|_{\rC^\alpha(\partial M)} \leq \varepsilon \| \varphi \|_W + C_\varepsilon \| \varphi \|_\infty 
\end{equation*}
for every $x \in D(W)$, i.e. $P$ is relatively $W$-bounded with bound $0$. 
\end{proof}

\begin{step}(Proof of \autoref{mainthm})
\end{step}
\begin{proof}
First we note that by \textbf{Step 5} we have
\begin{equation*}
\overline{N} = \overline{ W} - \overline{P} ,
\end{equation*}
and therefore using the \textbf{Steps 1, 3, 6} it follows that
\begin{equation}
N = \overline{N}|_{\dX} 
= ( \overline{W} - \overline{P})|_{\dX}
\supseteq \overline{W}|_{\dX} - P = W - P . \label{eq:1}
\end{equation}
On the other hand, 
by \textbf{Steps 2, 6} and \cite[Lem.~III.2.6]{EN:00}, $W - P$ generates
an analytic semigroup of angle $\nicefrac{\pi}{2}$ on $\dX$. Moreover, $\lambda
\in \rho(N) \cap \rho(W-P)$ for $\lambda$ large enough. This
implies equality in \eqref{eq:1} and hence the claim.  
\end{proof} 

\begin{cor}\label{N compact}
The Dirichlet-to-Neumann operator generates a compact semigroup on $\rC(\partial M)$. 
\end{cor}

\begin{proof}
By \eqref{W compact} the operator $W$ has compact resolvent. Since the Dirichlet-to-Neumann operator $N$ and $W$ differ only by a relatively bounded perturbation of bound $0$, it has compact resolvent by \cite[III.-(2.5)]{EN:00}. 
Hence the claim follows by  \autoref{mainthm} and \cite[Thm.~II.4.29]{EN:00}. 
\end{proof}

\begin{rem}\label{remark beta}
We can insert a strictly positive function $0<\beta \in\rC(\partial M)$ and consider $\tilde{B} := \beta \cdot B$. 
Then by multiplicative perturbation theory (cf. \cite[Sect.~III.1]{Hol:92}) the same generation result as above holds true.  

\end{rem}

\subsection{The Laplace-Beltrami operator with Wentzell boundary conditions}
\

In this subsection we study the Laplace-Beltrami operator with Wentzell boundary conditions and prove that it generates an analytic semigroup of angle $\nicefrac{\pi}{2}$ on $X = \rC(\overline{M})$.
To show this, we verify the assumptions of \cite[Thm.~3.1]{BE:18}.

\begin{lem}\label{B rel. A_0 bounded}
The feedback operator $B$ is relatively $A_0$-bounded with bound $0$. 
\end{lem}
\begin{proof}
By \cite[Chap.~5., Thm.~1.3]{Tay:96} 
 and the closed graph theorem we obtain
\begin{align*}
[D(A_0)] \hookrightarrow \rW^{2,p}(M) .
\end{align*}
Rellich's embedding (see \cite[Thm.~6.2, Part III.]{Ada:75}) implies
\begin{equation*}
\rW^{2,p}(M) \stackrel{c}{\hookrightarrow}\rC^{1,\alpha}(M) \stackrel{c}{\hookrightarrow}\rC^1(\overline{M})
\end{equation*}
for $p > \frac{m-1}{1-\alpha}$, so we obtain
\begin{align*}
[D(A_0)] \stackrel{c}{\hookrightarrow}\rC^1(\overline{M}) \hookrightarrow \rC(\overline{M}) . 
\end{align*}
Therefore, by Ehrling's lemma (cf. \cite[Thm.~6.99]{RR:04}), for every $\varepsilon >0$ there exists a constant $C_\varepsilon >0$ such that
\begin{equation*}
\| f \|_{\rC^1(\overline{M})} \leq \varepsilon \| f \|_{A_0} + C_\varepsilon \| f \|_X
\end{equation*}
for every $f \in D(A_0)$. Since $B \in \mathcal{L}(\rC^1(\overline{M}),\dX)$,
this implies the claim. 
\end{proof}

Now we prove the generator result for the operator with Wentzell boundary conditions.

\begin{thm}\label{Gen Wentzell}
The operator $A^B$ with Wentzell boundary conditions given by \eqref{eq:W-BC} for \eqref{L-B} and \eqref{nDeriv} generates a compact and analytic semigroup of angle $\nicefrac{\pi}{2}$ on $X = \rC(\overline{M})$. 
\end{thm}

\begin{proof}
We verify the assumptions from \cite[Thm.~3.1]{EF:05}.
The operator $A_0$ with Dirichlet boundary conditions is sectorial of angle $\nicefrac{\pi}{2}$ with compact resolvent by \cite[Thm.~2.8]{Bin:18a} and \cite[Cor.~3.4]{Bin:18a}. Moreover 
the Dirichlet operator $L_0$ exists by \autoref{L_0 exists} and the feedback operator $B$ is relatively $A_0$-bounded of bound $0$ by
\autoref{B rel. A_0 bounded}.
Lastly, the Dirichlet-to-Neumann operator $N$ generates a compact and analytic semigroup of angle $\nicefrac{\pi}{2}$ on $\rC(\partial M)$ by \autoref{mainthm} and \autoref{N compact}. Now the claim follows from \cite[Thm.~3.1]{EF:05}.
\end{proof}

\begin{rem}
As in \autoref{remark beta} we can insert a strictly positive, continuous
function $\beta > 0$ and the same result as \autoref{Gen Wentzell} becomes true.
\end{rem}

\section{Strictly elliptic operators on continuous functions on a compact manifold with boundary}

In this section we consider strictly elliptic second-order differential operators
with generalized Wentzell boundary conditions on $\tilde{X} := \rC(\overline{M})$ for a smooth, compact, orientiable, Riemannian manifold $(\overline{M},g)$
with smooth boundary $\partial M$. 
To this end, we take real-valued functions 
\begin{equation}
a_j^k = a_k^j \in \rC^{\infty}(\overline{M})
, \quad b_j \in \rC_c({M}), \quad c,d \in \rC(\overline{M}) \quad 1 \leq j,k \leq n, 
\label{coeff}
\end{equation}
satisfying the strict ellipticity condition 
\begin{equation*}
a^k_j(q) g^{jl}(q) X_k(q) X_l(q) > 0 \quad 
\end{equation*}
for all co-vectorfields $X_k, X_l$ on $\overline{M}$ with $(X_1(q),\dots, X_n(q)) \not = (0, \dots, 0)$. 
Then we define the maximal operator in divergence form as
\begin{align}
\tilde{A}_m f &:=
\sqrt{|a|} \div_g \left( \frac{1}{\sqrt{|a|}} a \nabla_M^g f \right) +  \langle b, \nabla_M^g f \rangle + c f,\\ 
D(\tilde{A}_m) &:= \biggl\{ f \in \bigcap_{p > 1} \rW^{2,p}_{\loc}(M) \cap \rC(\overline{M}) \colon \tilde{A}_m f \in \rC(\overline{M}) \biggr\}. \label{Def: A_m M}
\end{align}
As feedback operator we take
\begin{equation}
\tilde{B}f := - g(a \nabla_M^g f, \nu_g ) + d L f, \quad
D(\tilde{B}) := \biggl\{ f \in \bigcap_{p > 1} \rW^{2,p}_{\loc}(M) \cap \rC(\overline{M}) \colon \tilde{B}f \in \rC(\partial M) \biggr\}. \label{Def: B M}
\end{equation}
Corresponding to $L$ we choose $\partial \tilde{X} := \rC(\partial M^g)$.

\smallskip

The key idea is to reduce the strictly elliptic operator and the conormal derivative on $\overline{M}$, equipped by $g$, to the Laplace-Beltrami operator and to the normal derivative on $\overline{M}$, endowed by a new metric $\tilde{g}$. 

\smallskip

For this purpose we consider a $(2,0)$-tensorfield on $\overline{M}$ given by
\begin{align*}
\tilde{g}^{kl} = a^k_i g^{il} .
\end{align*}
Its inverse $\tilde{g}$ is a $(0,2)$-tensorfield on $\overline{M}$, which is a Riemannian metric since $a^k_j g^{jl}$ is strictly elliptic on $\overline{M}$. 
We denote $\overline{M}$ with the old metric by $\overline{M}^g$ and
with the new metric by $\overline{M}^{\tilde{g}}$ and remark that
$\overline{ M}^{\tilde{g}}$ is a smooth, compact, orientable Riemannian manifold with smooth boundary $\partial M$. 
Since the differentiable structures of $\overline{M}^g$ and $\overline{M}^{\tilde{g}}$ coincide, the identity
\begin{equation*}
\Id \colon \overline{M}^g \longrightarrow \overline{M}^{\tilde{g}}
\end{equation*}
is a $\rC^\infty$-diffeomorphism. Hence,
the spaces
\begin{align*}
X := \rC(\overline{ M}) &:=\rC(\overline{M}^{\tilde{g}}) =\rC(\overline{M}^g) = \tilde{X} \\
\text{and } \quad
\partial X := \rC(\partial M) &:=\rC(\partial M^{\tilde{g}}) =\rC(\partial M^g) = \partial \tilde{X} 
\end{align*} 
coincide.
Moreover, \cite[Prop.~2.2]{Heb:00} implies that the spaces
\begin{align}
\rL^p(M) &:= \rL^p(M^{\tilde{g}}) = \rL^p(M^g), \notag \\
\rW^{k,p}(M) &:= \rW^{k,p}(M^{\tilde{g}}) = \rW^{k,p}(M^g), \label{sobolev spaces} \\
\rL^p_{loc}(M) &:= \rL^p_{loc}(M^{\tilde{g}}) = \rL^p_{loc}(M^g), \notag \\
\rW^{k,p}_{loc}(M) &:= \rW^{k,p}_{loc}(M^{\tilde{g}}) = \rW^{k,p}_{loc}(M^g) \notag
\end{align} 
for all $p > 1$ and $k \in \N$ coincide.
We now denote by $A_m$ and $B$ the operators defined as in \textbf{Section 3} with respect to $\tilde{g}$. Moreover we denote $\hat{A}_m$ the operator defined in \eqref{Def: A_m M} for $b_k = c = 0$.

\subsection{The associated Dirichlet-to-Neumann operator and the Robin problem}
\

In this subsection we study the Dirichlet-to-Neumann operator $N^{\tilde{A}_m, \tilde{B}}$ associated with $\tilde{A}_m$ and $\tilde{B}$. 
First we prove that the generator properties of the Dirichlet-to-Neumann operators associated with $(\tilde{A}_m,\tilde{B})$ and $(A_m, B)$ are closely related.

\begin{lem}\label{Stoerung}
The operators $\hat{A}_m$ and $\tilde{A}_m$ differ only by a relatively $A_m$-bounded
perturbation of bound $0$. 
\end{lem}
\begin{proof}
From \eqref{sobolev spaces} we define
\begin{align*}
P_1 f := b_l g^{kl} \partial_k f 
\end{align*}
for $f \in D(A_m) \cap D(\hat{A}_m)$.
Morreys embedding (cf. \cite[Chap.~V. and Rem.~5.5.2]{Ada:75}) implies
\begin{align}
\bigl[D(\hat{A}_m)\bigr] \stackrel{c}{\hookrightarrow} \rC^1({M}) \hookrightarrow \rC(M).
\label{Embeddings}
\end{align}
Since $b_l \in \rC_c(M)$ we obtain
\begin{align*}
	\| P_1f \|_{\rC(\overline{M})} &\leq \sup_{q \in \overline{M}} | b_l(q) g^{kl}(q) (\partial_k f)(q) |\\
	&= \sup_{q \in {M}} | b_l(q) g^{kl}(q) (\partial_k f)(q) | \\
	&\leq C \sum_{k = 1}^n \| \partial_k f \|_{\rC(M)}
\end{align*}
and therefore $P_1 \in \mathcal{L}(\rC^1({M}), \rC(\overline{M}))$. Hence $D(\hat{A}_m)=D(\tilde{A}_m)$. By \eqref{Embeddings} we conclude from Ehrling's Lemma
(see \cite[Thm.~6.99]{RR:04}) that
\begin{align*}
	\| P_1f \|_{\rC(\overline{M})} \leq C \| f \|_{\rC^1(M)} &\leq \varepsilon \| \hat{A}_m f \|_{\rC(\overline{M})} + \varepsilon \| f \|_{\rC(\overline{M})} + C(\varepsilon) \| f \|_{\rC(M)} \\
	&\leq \varepsilon \| \hat{A}_m f \|_{\rC(\overline{M})} + \tilde{C}(\varepsilon) \| f \|_{\rC(\overline{M})}
\end{align*}
for $f \in D(\hat{A}_m)$ and all $\varepsilon > 0$ and hence $P_1$ is relatively $A_m$-bounded of bound $0$. Finally, remark that
\begin{align*}
	P_2 f := c \cdot f, \quad D(P_2) := \rC(\overline{M})
\end{align*}
is bounded and that
\begin{align*}
	\tilde{A}_m f = \hat{A}_m f + P_1 f + P_2 f
\end{align*}
for $f \in D(\hat{A}_m)$.
\end{proof}

\begin{lem}\label{LB}
	The operator $\hat{A}_m$ equals to the Laplace-Beltrami operator $\Delta^{\tilde{g}}_m$.
\end{lem}
\begin{proof}
	We calculate in local coordinates
	\begin{align*}
	\hat{A}_m f &= \frac{1}{\sqrt{|g|}} \sqrt{|a|} \partial_j \left(\sqrt{|g|} \frac{1}{\sqrt{|a|}} a_l^j g^{kl} \partial_k f\right) \\
	&= \frac{1}{\sqrt{|\tilde{g}|}} \partial_j \left(\sqrt{|\tilde{g}|} \tilde{g}^{kl} \partial_k f\right) = \Delta^{\tilde{g}}_m f
	\end{align*}
	for $f \in D(\hat{A}_m)= D(\Delta^{\tilde{g}}_m)$, since $|g| = |a| \cdot |\tilde{g}|$.
\end{proof}

\begin{lem}\label{Stoerung2}
The operators $B$ and $\tilde{B}$ differ only by a bounded perturbation.
\end{lem}
\begin{proof}
Since the Sobolev spaces coincide, we compute in local coordinates
\begin{align*}
\tilde{B}f &= - g_{ij} g^{jl} a_l^k \partial_k f g^{im} \nu_m + d L f \\
&= - g_{ij} \tilde{g}^{jl} \partial_k f g^{im} \nu_m + b_0 L f \\
&= - \tilde{g}_{ij} \tilde{g}^{jl} \partial_k f \tilde{g}^{im} \nu_m + d L f \\ 
&= B f + d L f
\end{align*}
for $f \in D(B)$. Since $d \cdot L f \in \rC(\partial M)$ we obtain $D(B) = D(\tilde{B})$ and 
$B$ and $\tilde{B}$ differ only by the bounded perturbation $d \cdot L$. 
\end{proof}

\begin{lem}\label{N^A_m Pert}
The Dirichlet-to-Neumann operator
$N^{\tilde{A}_m, \tilde{B}}$ associated with $\tilde{A}_m$ and $\tilde{B}$ generates a compact and analytic semigroup of angle $\alpha > 0$ on $\dX$ if and only if $N^{A_m, B}$ associated with $A_m$ and $B$ does so. 
\end{lem}
\begin{proof}
Let $P$ be the perturbation defined in the proof of \autoref{Stoerung}.
By \autoref{Stoerung} $P$ is relatively $A_m$-bounded of
bound $0$. Moreover, $\tilde{B}$ and $B$ only differ by a bounded perturbation by \autoref{Stoerung2}. 
Hence, the claim follows by \cite[Prop.~4.7]{BE:18}.
\end{proof}

\begin{thm}\label{D-N mainthm}
The Dirichlet-to-Neumann operator $N^{\tilde{A}_m, \tilde{B}}$ given by \eqref{def:N} for \eqref{Def: A_m M} and \eqref{Def: B M} generates
a compact and analytic semigroup of angle $\nicefrac{\pi}{2}$ on $X = C(\partial{M})$. 
\end{thm}

\begin{proof}
The claim follows by \autoref{mainthm} and \autoref{N^A_m Pert}.
\end{proof}

\begin{rem}
As in \autoref{remark beta} we can insert a strictly positive, continuous
function $\beta > 0$ and the same result as \autoref{mainthm} becomes true.
\end{rem}

\begin{rem}
\autoref{D-N mainthm} improves and generalizes the main result in \cite{Esc:94}. 
If we consider $M = \Omega \subset \R^n$ equipped with the euclidean metric $g = \delta$, we obtain the maximal angle $\nicefrac{\pi}{2}$ of analyticity in this case. This is the main result in \cite{EO:17} for smooth coefficients.
\end{rem}

Now we use \autoref{D-N mainthm} to obtain existence and uniqueness for the associated Robin problem \eqref{Robin Problem}. Moreover, we obtain a maximum principle for the solutions
of these problems.

\begin{cor}[Existence, uniqueness and maximum principle for the general Robin problem]
There exists $\omega \in \R$ such that for
all $\lambda \in \C \setminus (-\infty, \omega)$ the problem \eqref{Robin Problem} has a unique solution $u \in D(A_m) \cap D(B)$. 
This solution satisfies the maximum principle
\begin{align*}
|\lambda| \max_{p \in \overline{M}} |u(p)| \leq C |\lambda|\max_{p \in \partial{M}} |u(p)| 
= C |\lambda | \| Lu \|_{\dX} \leq \tilde{C} \| \varphi \|_{\dX} = \tilde{C} \max_{p \in \partial{M}} |\varphi(p)| .
\end{align*}
\end{cor}

\begin{proof}
The existence and uniqueness follows immediately by \autoref{D-N mainthm}.
The first inequality is the interior maximum principle. The second inequality is a direct consequence from \autoref{D-N Resolvente} and \autoref{D-N mainthm}.
\end{proof}

\subsection{The associated operator $\tilde{A}^{\tilde{B}}$ with Wentzell boundary conditions}

\begin{lem}\label{W Pert}
The operator $\tilde{A}^{\tilde{B}}$ generates a compact and analytic semigroup of
angle $\alpha > 0$ on $X$ if and only if $A^B$ does. 
\end{lem}
\begin{proof}
As seen in the proof of \autoref{N^A_m Pert}, the operators
$A_m$ and $\tilde{A}_m$ differ only by a relatively $A_m$-bounded perturbation with bound $0$ while $B$ and $\tilde{B}$ differ only by a bounded perturbation.
Therefore, the claim follows by \cite[Thm.~4.2]{BE:18}.
\end{proof}

\begin{thm}\label{W mainthm}
The operator $\tilde{A}^{\tilde{B}}$ given by \eqref{eq:W-BC} for \eqref{Def: A_m M} and \eqref{Def: B M} generates a compact and analytic semigroup of angle $\nicefrac{\pi}{2}$ on $X = \rC(\overline{M})$.
\end{thm}

\begin{proof}
The claim follows by \autoref{Gen Wentzell} and \autoref{W Pert}.
\end{proof}

\begin{rem}
As in \autoref{remark beta} we can insert a strictly positive, continuous
function $\beta > 0$ and the same result as \autoref{W mainthm} becomes true.
\end{rem}

\begin{rem}
\autoref{W mainthm} improves and generalizes \cite[Cor.~4.5]{EF:05}. 
If we consider $M = \Omega \subset \R^n$ equipped with the euclidean metric $g = \delta$, we obtain the maximal angle $\nicefrac{\pi}{2}$ of analyticity.  
\end{rem}

\begin{cor}
By \autoref{W mainthm} the initial boundary problem
\begin{align*}
\begin{cases}
\frac{d}{dt} u(t, p) = \sqrt{|a(p)|} \div_g \left( \frac{1}{\sqrt{|a(p)|}} a (p) \nabla_M^g u(t,p) \right) +  \langle b (p), \nabla_M^g u(t,p) \rangle + c(p) u(t,p)
&\text{ for } t \geq 0, p \in \overline{M}, \\
\frac{d}{dt} u(t, p)
= - \beta g(a (p) \nabla_M^g u(t, p), \nu_g (p) ) + d(p) u(t, p)
&\text{ for } t \geq 0, p \in \partial M, \\
\ \ u(0, p) = u_0(p) &\text{ for } p \in \overline{M}
\end{cases}
\end{align*}
for $a, b, c, d$ as in \eqref{coeff}, $\beta > 0$ and $u_0(p) \in D(A^B)$
has a unique solution on $\rC(\overline{M})$. This solution is governed by an analytic semigroup in the right half-plane. 
\end{cor}

Finally, we consider the elliptic problem
\begin{align}
\begin{cases}
A_m f - \lambda f = h \\
LA_m f = B f ,
\end{cases} \label{Elliptic Wentzell}
\end{align}
for $f \in D(A_m) \cap D(B)$ and $h \in X = \rC(\overline{M})$. Then the following holds.

\begin{cor}
There exists $\omega \in \R$ such that for
all $\lambda \in \C \setminus (-\infty, \omega)$ the problem \eqref{Elliptic Wentzell} has a unique solution $u \in D(A_m) \cap D(B)$. 
This solution satisfies the maximum principle
\begin{align*}
|\lambda| \max_{p \in \overline{M}} |u(p)| = |\lambda | \| u \|_{X} \leq C \| h \|_{X} = C \max_{p \in \overline{M}} |h(p)| .
\end{align*}
\end{cor}
\begin{proof}
This follows immediately by \autoref{W mainthm}.
\end{proof}


\bigskip

\newcommand{\etalchar}[1]{$^{#1}$}

\bigskip
\emph{Tim Binz}, University of Tübingen, Department of Mathematics, Auf der Morgenstelle 10, D-72076 Tübingen, Germany,
\texttt{tibi@fa.uni-tuebingen.de}

\end{document}